\documentclass[11pt]{amsart}
\usepackage{amssymb,amsmath,epsfig,mathrsfs, enumerate, xparse, mathtools}
\usepackage[nodisplayskipstretch]{setspace}
\usepackage{graphicx}
\usepackage[normalem]{ulem}
\usepackage{fancyhdr}
\usepackage[margin=2.5cm]{geometry}
\usepackage{hyperref}
\hypersetup{
 colorlinks   = true,
 urlcolor     = blue,
 linkcolor    = blue,
 citecolor   = red ,
 bookmarksopen=true
}

\theoremstyle{plain}
\newtheorem{thrm}{Theorem}[section]
\newtheorem{lemma}[thrm]{Lemma}

\newtheorem{rmrk}[thrm]{Remark}
\newtheorem{dfn}[thrm]{Definition}

\allowdisplaybreaks
\begin{document}
\newcommand{\sn}{\mathbb{S}^{n-1}}
\newcommand{\SL}{\mathcal L^{1,p}( D)}
\newcommand{\Lp}{L^p( Dega)}
\newcommand{\CO}{C^\infty_0( \Omega)}
\newcommand{\Rn}{\mathbb R^n}
\newcommand{\Rm}{\mathbb R^m}
\newcommand{\R}{\mathbb R}
\newcommand{\Om}{\Omega}
\newcommand{\Hn}{\mathbb H^n}
\newcommand{\aB}{\alpha B}
\newcommand{\eps}{\ve}
\newcommand{\BVX}{BV_X(\Omega)}
\newcommand{\p}{\partial}
\newcommand{\IO}{\int_\Omega}
\newcommand{\bG}{\boldsymbol{G}}
\newcommand{\bg}{\mathfrak g}
\newcommand{\bz}{\mathfrak z}
\newcommand{\bv}{\mathfrak v}
\newcommand{\Bux}{\mbox{Box}}
\newcommand{\e}{\ve}
\newcommand{\X}{\mathcal X}
\newcommand{\Y}{\mathcal Y}
\newcommand{\W}{\mathcal W}
\newcommand{\la}{\lambda}
\newcommand{\vf}{\varphi}
\newcommand{\rhh}{|\nabla_H \rho|}
\newcommand{\Ba}{\mathcal{B}_\beta}
\newcommand{\Za}{Z_\beta}
\newcommand{\ra}{\rho_\beta}
\newcommand{\na}{\nabla_\beta}
\newcommand{\vt}{\vartheta}

\numberwithin{equation}{section}

\newcommand{\RN} {\mathbb{R}^N}
\newcommand{\Sob}{S^{1,p}(\Omega)}
\newcommand{\Dxk}{\frac{\partial}{\partial x_k}}
\newcommand{\Co}{C^\infty_0(\Omega)}
\newcommand{\Je}{J_\ve}
\newcommand{\beq}{\begin{equation}}
\newcommand{\bea}[1]{\begin{array}{#1} }
\newcommand{\eeq}{ \end{equation}}
\newcommand{\ea}{ \end{array}}
\newcommand{\eh}{\ve h}
\newcommand{\Dxi}{\frac{\partial}{\partial x_{i}}}
\newcommand{\Dyi}{\frac{\partial}{\partial y_{i}}}
\newcommand{\Dt}{\frac{\partial}{\partial t}}
\newcommand{\aBa}{(\alpha+1)B}
\newcommand{\GF}{\psi^{1+\frac{1}{2\alpha}}}
\newcommand{\GS}{\psi^{\frac12}}
\newcommand{\HFF}{\frac{\psi}{\rho}}
\newcommand{\HSS}{\frac{\psi}{\rho}}
\newcommand{\HFS}{\rho\psi^{\frac12-\frac{1}{2\alpha}}}
\newcommand{\HSF}{\frac{\psi^{\frac32+\frac{1}{2\alpha}}}{\rho}}
\newcommand{\AF}{\rho}
\newcommand{\AR}{\rho{\psi}^{\frac{1}{2}+\frac{1}{2\alpha}}}
\newcommand{\PF}{\alpha\frac{\psi}{|x|}}
\newcommand{\PS}{\alpha\frac{\psi}{\rho}}
\newcommand{\ds}{\displaystyle}
\newcommand{\Zt}{{\mathcal Z}^{t}}
\newcommand{\XPSI}{2\alpha\psi \begin{pmatrix} \frac{x}{\left< x \right>^2}\\ 0 \end{pmatrix} - 2\alpha\frac{{\psi}^2}{\rho^2}\begin{pmatrix} x \\ (\alpha +1)|x|^{-\alpha}y \end{pmatrix}}
\newcommand{\Z}{ \begin{pmatrix} x \\ (\alpha + 1)|x|^{-\alpha}y \end{pmatrix} }
\newcommand{\ZZ}{ \begin{pmatrix} xx^{t} & (\alpha + 1)|x|^{-\alpha}x y^{t}\\
     (\alpha + 1)|x|^{-\alpha}x^{t} y &   (\alpha + 1)^2  |x|^{-2\alpha}yy^{t}\end{pmatrix}}
\newcommand{\norm}[1]{\lVert#1 \rVert}
\newcommand{\ve}{\varepsilon}
\newcommand{\D}{\operatorname{div}}

\title[backward uniqueness etc.]{ Strong Backward uniqueness for sublinear parabolic equations}

\author{Vedansh Arya}
\address{Tata Institute of Fundamental Research\\
Centre For Applicable Mathematics \\ Bangalore-560065, India}\email[Vedansh Arya]{vedansh@tifrbng.res.in}

\author{Agnid Banerjee}
\address{Tata Institute of Fundamental Research\\
Centre For Applicable Mathematics \\ Bangalore-560065, India}\email[Agnid Banerjee]{agnidban@gmail.com}

\thanks{Second author is supported in part by SERB Matrix grant MTR/2018/000267 and by Department of Atomic Energy,  Government of India, under
project no.  12-R \& D-TFR-5.01-0520.}


%
%
%

\maketitle
\begin{abstract}
In this paper, we establish strong backward uniqueness for solutions to  sublinear parabolic equations of the type \eqref{e0}.  The proof of our  main result Theorem \ref{main} is achieved by means of a  new  Carleman  estimate and  a Weiss type monotonicity  that are    tailored for such parabolic sublinear operators.

\end{abstract}

\tableofcontents
 
\section{Introduction and the statement of the main result}
The purpose of this work is to establish strong backward uniqueness for parabolic sublinear equations of the type 
\begin{equation}\label{e0}
\D(A(x, t) \nabla v) + v_t + W v + h((x,t), v) =0
\end{equation}
 where 
\begin{equation}\label{w}
\norm{W}_{L^\infty} \le M
\end{equation}
and the matrix $A$ is  symmetric, uniformly elliptic  and satisfies
\begin{equation}
    \lambda |\xi|^2 \le a^{ij}\xi_i \xi_j \le \Lambda |\xi|^2 \ \forall \xi \in \R^n
\end{equation}
\begin{equation}\label{coef}
    |\nabla a_{ij}(x,t)| \le \frac{M}{1 + |x|}, |\partial_t a_{ij}(x,t)| \leq M.
\end{equation} 
On  the sublinear term $h$, we assume the following.
\begin{align}\label{str}
&h\left ((x,t), 0 \right ) = 0, \\
&H\left ( (x,t), s \right ) = \int_0^s h(X,s) ds,\notag \\
&0 < s h \left ( (x,t) , s \right ) \le q H \left ( (x,t), s \right ) \text{ for some } q \in [1,2),  \notag\\
& H((x,t), s) \geq \ve_0\ \text{for all $|s|>1$ and some  $\ve_0>0$},\notag\\
&|\nabla_xH| \le \frac{C_0}{1+|x|} H,  \notag \\
&|\partial_tH| \le {C_0} H,\notag \\
&h\left ( (x,t) , s \right )  \le C_0 \sum_{i=1}^m |s|^{p_i -1} \text{ for   $p_i$'s  $\in [1,2)$  and some $C_0>0$}.  \notag
\end{align}
We note that from  \eqref{str} it follows that given $L>0$, there exists    $c_0= c_0(L) >0$, such that 
\begin{equation}\label{str1}
 H((x,t), s) \geq c_0 |s|^{q}\ \text{for $|s| < L$}.
\end{equation}

A  prototypical $h$ satisfying \eqref{str}  is given by 
  \[
  h((x,t), v) = \sum_{i=1}^l c_i(x,t) |v|^{p_i-2} v,\] where for each $i$,  $p_i \in [1,2)$,  $0<k_0<c_i< k_1$, $|\nabla_x c_i| < \frac{C_0}{1+|x|}$ and $|\partial_t c_i| < C_0$
 for some $k_0, k_1\ \text{and}\ C_0$.  In this case, we can take  $q= \text{max}\{p_i\}$.  In order to put things in the right perspective, we note that  motivated by the study of nonlinear eigenvalue problems as well as  the analysis of corresponding nodal domains as in \cite{PW} and also  because of certain connections to porous media type equations (as in \cite{Vaz}), Soave  and Weth in \cite{SW}   established weak unique continuation  for equations of the type
 \begin{equation}\label{t100}
 \operatorname{div}(A(x) \nabla v) + h(x, v) + Wv=0
 \end{equation}
 Such equations are modeled on
 \begin{equation}\label{s10}
 -\Delta v = |v|^{p-2} v.
 \end{equation}
 Note that  the study of strong unique continuation for  \eqref{s10} cannot be reduced to that for
 \[
 -\Delta + W
 \]
  because in this case, $W= |v|^{p-2}$ need not be in $L^p$  for any $p$ near the zero set of $v$  as $p \in (1, 2)$.  In fact  such sublinear equations have their intrinsic  difficulties  and this is also partly visible from the fact  that the sign assumption on the sublinearity $h$ in \eqref{str} is quite crucial because otherwise unique continuation fails. This later fact follows  from  a counterexample in \cite{SW} where it is shown that unique continuation is not true for
 \begin{equation}\label{s100}
 \Delta v= |v|^{p-2} v, \ p \in (1, 2).
 \end{equation}
    In \cite{SW}, the authors adapted the frequency function approach  of Garofalo and Lin  as in \cite{GL}.   The  question of strong unique continuation for such  sublinear equations  was then later  addressed by Ruland in \cite{Ru} via new Carleman estimates for  such sublinear elliptic operators. Such a result was generalized to degenerate Baouendi-Grushin type operators by one of us with Garofalo and Manna in \cite{BGM}.   We also  refer to   the  interesting work of Soave and Terracini in \cite{ST},  where the authors study the following two phase membrane   problem
    \begin{equation}\label{t10}
    -\Delta v= \lambda_{+} (v^{+})^{q-1} - \lambda_{-} (v^{-})^{q-1}, \ \text{where $\lambda_{+}, \lambda_{-} >0$, $q \in [1, 2)$}
    \end{equation}
    and established  strong unique continuation property as well as  a regularity result for the nodal domains.  The key object in their analysis was a  new monotonicity formula for a $2$-parameter family of Weiss type functionals first  introduced by Weiss in  his seminal work \cite{We} in the context of classical obstacle problem.   The space like strong unique continuation for such backward parabolic sublinear operators as in \eqref{e0} has been  obtained more recently by one of us with Manna in \cite{BM} by    a generalization of the  Carleman estimates in \cite{EF}, \cite{EFV} to the sublinear situation.   In this paper, we complement the existing unique continuation results  for such sublinear operators  by establishing a strong backward uniqueness result in  this framework. Our main result Theorem \ref{main} generalizes the backward uniqueness result  recently obtained by Wu and Zhang in \cite{WZ} for linear equations with similar structural assumptions( see also \cite{WZ1}).   Similar to that in \cite{BM}, the proof of   our result relies on a careful generalization of the  Carleman estimate in \cite{WZ} to the sublinear case. Also  when the principal part $A=\mathbb{I}$, we obtain a  strong backward uniqueness result when the structure condition \eqref{str} holds for  $p_i, q$ in the bigger range $[1, 2)$( see Theorem \ref{main} ii) below). This is achieved  by means of a new parabolic Weiss type monotonicity formula  which generalizes the one in the elliptic case due to Soave and Terracini in \cite{ST}.   As the reader will see, the proof of both the Carleman estimate as well as the Weiss type monotonicity  is made possible by  a combination of several non-trivial geometric facts which  thanks  to the specific structure of the sublinearity,  beautifully combine.     Moreover compared to that in \cite{BM}, some  new challenges appear   in the proof of the  Carleman estimate  \eqref{ce1} below because the weights involved are different  from that in \cite{BM} and  consequently  our proof crucially relies on     new  inequalities  in such weighted spaces. This    constitutes one of the novelties of our work. We also  believe that the new  parabolic Weiss type monotonicity that we obtain in this framework also has an independent interest. 
We now state our main result.

\begin{thrm}\label{main}
Let $v$ be a solution to the backward parabolic sublinear equation  \eqref{e0} in $\Rn \times [0,1]$. Then the following backward uniqueness results hold.

\medskip

i)  Assume that $v$ satisfies the following Tychonoff type growth assumption
\begin{equation}\label{ty}
|v(x,t)| \leq Ne^{N|x|^2}
\end{equation}
for some $N>0$.
Also assume that   matrix $A$ satisfies the derivative bounds as in   \eqref{coef}, the potential $W$ satisfies the bound in \eqref{w} and the structure condition \eqref{str} holds for $p_i'$s and $ q \in (1,2)$.  Now if  $v$ vanishes to infinite order in the sense of \eqref{vp} at $(0,0)$, then $v \equiv 0$.

\medskip

ii) Let $A= \mathbb{I}$, $v$ be bounded  and assume that $h$ satisfies the structure conditions in \eqref{str} for  $p_i$'s and $q \in [1, 2)$.  Now if $v$ vanishes to infinite order in the sense of \eqref{vp1} at $(0,0)$, then $v \equiv 0$.

\end{thrm}
Before proceeding further, we make a couple of  important remarks.
\begin{rmrk}
We note that it is shown in \cite{BM} that the two notions of vanishing order \eqref{vp} and \eqref{vp1} coincide when $v$ solves \eqref{e0}  with $h$ satisfying \eqref{str} for $p_i, q \in (1,2)$.  We also observe that unlike Theorem \ref{main} i), the result in Theorem \ref{main} ii) covers the  end point  case  $q=1$ when $A=\mathbb{I}$. The  proof of this later fact is based on a new  Weiss type monotonicity  which works for any generic sublinearity satisfying \eqref{str}  and thus  provides an alternate proof of the backward uniqueness result in i) when the  principal part is the Laplacian. A typical scenario where the result in   Theorem \ref{main} ii) applies is say in the case  of a     parabolic two phase membrane problem of the type,
   \[
    -\Delta v=  v_t  + \lambda_{+} (v^{+})^{q-1} - \lambda_{-} (v^{-})^{q-1}, \ \text{where $\lambda_{+}, \lambda_{-} >0$, $q \in [1, 2)$.}\]
    Over here, we note that when $q=1$, $(v^+)^{q-1} \coloneqq \chi_{\{v>0\}}$ and $(v^-)^{q-1} \coloneqq \chi_{\{v<0\}}$.   It remains to be seen whether Theorem \ref{main} ii) continues to be valid for more general $A$ satisfying \eqref{coef}. In this regard,  we would  however like to  mention that to the best of our knowledge, Weiss type monotonicity is not known even for linear parabolic equations with variable coefficient principal part satisfying \eqref{coef}.     \end{rmrk}

    \begin{rmrk}
    We would also like to remark that a growth condition of the type \eqref{ty} is needed for backward uniqueness to hold. This follows from  a counterexample due to Frank Jones in \cite{J} where it is shown that there exists a non-trivial  unbounded caloric function that is supported in a time strip of the type $\R^n \times (t_1, t_2)$. Also from an example as   in \cite{WZ1}, it follows that  the decay assumption on  the derivatives of the principal part  as in \eqref{coef} is somewhat  optimal as well. We note that such a decay is related to the exponential growth rate of the solution as in \eqref{ty}. 
    \end{rmrk}
We finally note that the subject of strong unique continuation and backward uniqueness has a long history and several ramifications by now. We refer the reader to  \cite{Al}, \cite{AKS}, \cite{Car}, \cite{Ch}, \cite{EF}, \cite{EFV}, \cite{GL}, \cite{Ho},  \cite{JK}, \cite{KT0}, \cite{KT}, \cite{L},\cite{LM},  \cite{LO}, \cite{Mi}, \cite{Po}, \cite{V} and  \cite{Y}  and one can find other references therein. 

The paper is organized as follows. In Section \ref{s:n}, we introduce some basic  notations   and also  gather some known results that are relevant to our present work. In Section \ref{s:main}, we prove our new Carleman estimate and  a Weiss type monotonicity  and consequently establish our backward uniqueness results.

\section{Notations and Preliminaries}\label{s:n}

  A generic point $(x,t)$ in space time $\Rn \times (0, \infty)$ will be denoted by $X$. For notational convenience, $\nabla f$ and  $\operatorname{div}\ f$ will respectively refer to the quantities  $\nabla_x f$ and $ \operatorname{div}_x f$ of a given function $f$.   The partial derivative in $t$ will be denoted by $\p_t f$ and also by $f_t$. The partial derivative $\partial_{x_i} f$  will be denoted by $f_i$. We indicate with $C_0^{\infty}(\Omega)$ the set of compactly supported smooth functions in the region $\Omega$  in space-time.  Also for $H$ as in \eqref{str}, $\partial_t H$ or $H_t$ will denote the derivative with respect to the variable $t$ of the function,
  \[
  t \to H((x,t), v)
  \]
  where $x$ and $v$ are treated as constants. Likewise, $\nabla_x H$ will denote the derivative with respect to the variable $x$ of the function
  \[
  x \to H((x,t), v)
  \]
  with $t$ and $v$ being constants. 
  
  We now define the relevant notions of vanishing to infinite order.

\begin{dfn}\label{vp}
We say that a function $u$ to vanishes to infinite order in space  at some $(x_0, t_0)$ if
 \begin{equation}\label{van}
 \begin{cases}
\text{ given $k>0$, there exists $C_k>0$ such that}
\\
|u(x,t_0)|\leq C_k|x-x_0|^k\ \text{as $x \to x_0$,}\end{cases}\end{equation}
\end{dfn}

Likewise, vanishing to infinite order in space-time is defined as follows.

\begin{dfn}\label{vp1}
We say that a function $u$ to vanishes to infinite order in space-time  at some $(x_0, t_0)$ if
\begin{equation}\label{van1}
\begin{cases}
\text{ given $k>0$, there exists $C_k>0$ such that}
\\
|u(x, t) | \leq C_k (|x - x_0|^2 + |t-t_0|^{1/2})^k\ \text{as $(x,t) \to (x_0, t_0)$}
\end{cases}
\end{equation}
\end{dfn}

We now state the relevant Rellich type identity from \cite{WZ} which is a slight generalization of the one in \cite{EF} and \cite{EFV}. See Lemma 3.1 in  \cite{WZ}.

\begin{lemma}\label{re1}
Suppose $F$ is differentiable, $F_0$ and $G$ are twice differentiable and $G>0$. Then, the following identity holds for any $u \in C_0^\infty(\R^n \times [0,T])$,
\begin{align}\label{re2}
 &\frac{1}{2} \int_{\R^n \times [0,T]} M_0 u^2G dx dt + \int_{\R^n \times [0,T]} \left ( 2D_G + ( \frac{\partial_t G - \tilde{\Delta}G}{G} - F ) A \right ) \nabla u \cdot \nabla u G dx dt 
 \\
 &- \int_{\R^n \times [0,T]} uA \nabla u \cdot \nabla (F- F_0) G dx dt \notag \\
    & = 2 \int_{\R^n \times [0,T]}Lu(Pu - Lu)G dx dt + \int_{\R^n} A \nabla u \cdot \nabla u G dx |_0^T + \frac 12 \int_{\R^n}u^2 FG dx |_0^T\notag
\end{align}
where 
\begin{align*}
\tilde \Delta &= \D(A(x)\nabla )\\
Pu &= \D(A(x,t) \nabla u) + u_t\\
    Lu &= u_t - A \nabla u \cdot \nabla \log G + \frac{F}{2}u \\
    M_0 &= \partial_t F + F \left ( \frac{ \partial _t G - \tilde{\Delta}G}{G} - F \right ) + \tilde{\Delta}{F_0} - A \nabla (F - F_0) \cdot \nabla \log G, 
    \end{align*}
and $D_G^{ij} = A^{ik} \partial_{kl} ( \log G ) a^{lj} + \frac{\partial_l (\log G)}{2}(a^{ik}\partial_k a^{lj} + a^{jk}\partial_k a^{li} - a^{kl}\partial_k a^{ij}) + \frac 12 \partial _t a^{ij}$
\end{lemma}
Similar to that in \cite{WZ}, we now let 
\begin{equation}\label{g}
G = e^{2\gamma(t^{-K} -1) - \frac{b \left< x \right>^2 + K}{t}}, \ \text{($K$ large enough to be chosen later)} \end{equation}
where $<x>= \sqrt{1+ |x|^2}$.
Then we observe that
\begin{equation}\label{g1}
\frac{\partial_tG - \Tilde{\Delta}G}{G} = \frac{b \left< x \right>^2 - 4 b^2 a^{ij}x_ix_j + K}{t^2} + \frac{2 b \left ( a^{ii} + \partial_k a^{kl}x_l \right )}{t}  - 2 \gamma \frac{K}{t^{K+1}}\end{equation}
Likewise, we define
\begin{align}\label{ff0}
    F & = \frac{b \left< x \right>^2 - 4 b^2  a^{ij}x_ix_j + K}{t^2} + \frac{2 b  a^{ii} - d}{t} - \frac{2\gamma K}{t^{K+1}}, \text{    $(d$ will be chosen in terms of $K$ as in Lemma \ref{bdd1} below)} \\
   {F_0} & = \frac{b \left< x \right>^2 - 4 b^2 a_{\ve}^{ij}x_ix_j + K}{t^2} + \frac{2 b a_{\ve}^{ii} - d}{t} - \frac{2\gamma K}{t^{K+1}}\notag
\end{align}
where $a_{\ve}^{ij}$ is the $\ve$ mollification of $a^{ij}$ corresponding to $\ve=1/2$ ( see lemma 3.2 in \cite{WZ}). We also need the following inequalities from \cite{WZ} corresponding to these choices of $G, F$ and $F_0$( see lemma 3.3 in \cite{WZ}).
\begin{lemma}\label{bdd1}
 Set $b = \frac{1}{8\Lambda}$ and $ d = \frac{K}{4}$. Then  for $ K \ge K_0(n, \Lambda, \lambda, M )$, we have 
\begin{align*}
    2D_G + \left (\frac{\partial_tG - \tilde{\Delta}G}{G} - F \right ) A & \ge \frac{\lambda K}{8t} I_n \\
    \partial_tF + F \left ( \frac{\partial_t G - \Tilde{\Delta}G}{G} - F \right ) & \ge \frac{b K \left< x \right>^2}{16t^3} \\
    |\tilde{\Delta}F_0| & \le \frac{C\left< x \right>^2}{t^2} \\
    |\nabla(F - F_0)| & \le \frac{C\left< x \right> }{t^2}
\end{align*}
where $C$ depends on $n, \lambda, \Lambda, M$ and is independent of $K$.
\end{lemma}
 Lemma \ref{bdd1} in particular implies that the following estimate holds for the quantity on the left hand side of the identity in \eqref{re2}. 
\begin{lemma}\label{bdd2}
With $G, F, F_0$ as in Lemma \ref{bdd1}, we have that the following integral estimate holds for all $K$ sufficiently large and  $u \in C_0^\infty(\R^n \times  (0,T))$,
\begin{align}\label{l1}
&\frac{1}{2}\int_{\R^n \times (0,T)} M_0 u^2 G dx dt + \int_{\R^n \times (0,T)} \left(2D_G + \left( \frac{\partial_tG - \tilde{\Delta}G}{G} - F \right )A \right ) \nabla u \cdot \nabla u G dx dt 
\\
&- \int_{\R^n \times (0,T)} u A \nabla u \cdot \nabla(F- F_0) G dx dt\notag\\
 & \ge  C K\left( \int_{\R^n \times (0,T)} u^2 \frac{ \left< x \right>^2}{t^3}G dx dt + \int_{\R^n \times (0,T)} \frac{|\nabla u|^2}{t} G dx dt \right)\notag
\end{align}
for some universal $C>0$. 
\end{lemma}
\begin{proof}
This can be deduced from the estimates $(41)-(44)$ in \cite{WZ} which in turns relies on the inequalities in Lemma \ref{bdd1} above.

\end{proof}

We also need the following  real analysis lemma from \cite{EF}.  See lemma 3.3 in \cite{EF}.
\begin{lemma}\label{rl}
 Given $m>0$, $\exists C_m$ such that for all $y \ge 0$ and $0 < \epsilon < 1$, 
 \[
 y^m e^{-y} \le C_m \left [ \epsilon + \left(\log(\frac{1}{\epsilon}) \right )^m e^{-y} \right ]
 \]
\end{lemma}
We now state our main Carleman estimate which is needed to prove the backward uniqueness result.
\begin{thrm}\label{ce}
Let  $u \in C_0^\infty(\R^n \times (0,T))$ be a solution of
\begin{equation}\label{o1}
\D(A(x,t) \nabla u) + u_t + h((x,t), u)  +Wu=g
\end{equation}
where $h$ satisfies the structure conditions in \eqref{str}. Then the following estimate holds with $G$ as in \eqref{g} for some universal $C>0$,
\begin{equation}\label{ce1}
 K \int ( u^2 + |\nabla u|^2 ) G dx dt + \gamma K \int \frac{HG}{t^{K+1}} dx dt \le C\left( \int He^{-2\gamma - \frac{\frac{b}{2} \left< x \right>^2 + K}{t} } dx dt + \int g^2 G dx dt\right) \end{equation}
where  $K, \gamma$ are  large enough depending  only on $n,\lambda, \Lambda, p_i,q, M, T$.
\end{thrm}
\section{Proof of the main results}\label{s:main}
We first establish the Carleman estimate  as in \eqref{ce1} which is needed to prove Theorem \ref{main} i).
\begin{proof}[Proof of Theorem \ref{ce}]
Let $G, F, F_0$ be as in Section \ref{s:n}. We start by applying the identity in Lemma \ref{re1} with $u$ as in the hypothesis.  Before proceeding further, we remark that in all the subsequent integrals, the measure $dxdt$ will be omitted. By  using the equation \eqref{o1} satisfied by $u$, we note that  the  corresponding right hand side in \eqref{re2} equals
\begin{align}\label{t0}
    & 2 \int Lu(Pu - Lu) G  \\
    &= 2\int \left ( u_t - A \nabla u \cdot \nabla \log G + \frac{Fu}{2} \right ) \left ( g - Wu - h(X,u) \right ) G - 2 \int (Lu)^2 G \notag\\
    & \le 2\int g^2 G dx dt +2 M \int u^2 G    -2\int \left(u_t - A \nabla u \cdot \nabla \log G + \frac{Fu}{2}\right) h(X,u) G \notag \\
    & \text{(using $2ab \leq a^2+ b^2$ with $a= Lu$ and $b=g-Wu$ and also that $b^2 \leq 2(g^2 +  M u^2)$)}\notag 
    \end{align}
    Now  by integrating by parts,  the last integral in \eqref{t0} ( which we denote by $I_1$)  can be expressed as
    \begin{align}\label{f1}
    & I_1= - 2\int [((H(X,u))_t - \partial_t H) - A \nabla H(X,u) \cdot \nabla \log G + A \nabla_x H \cdot \nabla \log G ] G  - \int Fu\ h(X,u) G  \\
    &=2 \int H(X,u) (G_t - \tilde{\Delta} G) - \int Fu\  h(X,u) G  +2 \int ( \partial_t H - A \nabla _x H \cdot \nabla \log G ) G \notag
\end{align}    
We note that in \eqref{f1} above, we rewrote  $A\nabla u\cdot \nabla \log G\ h(X,u)$ as 
\[
A\nabla u\cdot \nabla \log G\ h(X,u)=A \nabla H(X,u) \cdot \nabla \log G - A\nabla_x H \cdot \nabla \log G
\]
Then from \eqref{g1}  and the expression of  $F$ as in \eqref{ff0}, we observe that $I_1$  can be further rewritten as
\begin{align}\label{t1}
&I_1  = 2 \int ( \partial_t H - A \nabla _x H \cdot \nabla \log G ) G + 2 \int \left( H(X,u) - \frac{uh(X,u)}{2}\right) \left( \frac{b \left< x \right>^2 - 4 b^2 a_{ij}x_ix_j + K}{t^2}\right)G \\
&+2 \int \left( H(X,u) - \frac{uh(X,u)}{2}\right) \frac{2ba^{ii}}{t}G-2\int \left(H(X,u) - \frac{uh(X,u)}{2}\right) \frac{2\gamma K}{t^{K+1}}G\notag\\
&   + 2 \int \frac{H(X,u)}{t} \partial_k a^{kl}x_l G  + \int \frac{uh(X,u)d}{t}G\notag 
\end{align}
Now using the bounds on $\nabla_x H$ and $\partial_t H$ as in \eqref{str}, the bounds for $\nabla_x a^{ij}$ as in \eqref{coef}   and also that $ u h(X, u) \leq q H(X,u)$ for some $q \in (1,2)$, we obtain from \eqref{t1} that $I_1$ can be estimated from above in the following way,
\begin{equation}\label{k1}
 I_1 \le C \int\frac{HG}{t} + C \int \frac{H G \left< x \right>^2}{t^2}  + C \int  \frac{K}{t^2} HG + (q - 2) \int \frac{HG}{t^{K+1}}(2\gamma K)
\end{equation}
where $C$ is some universal constant independent of $K$. We note that the last integral in \eqref{k1} comes from estimating the integral   $2\int \left(H(X,u) - \frac{uh(X,u)}{2}\right) \frac{2\gamma K}{t^{K+1}}G$ in \eqref{t1} in the following way,
\[
2\int \left(H(X,u) - \frac{uh(X,u)}{2}\right) \frac{2\gamma K}{t^{K+1}}G \geq (2- q) \int \frac{HG}{t^{K+1}} (2 \gamma K).
\]
We would like to mention  that this is precisely the place where we use the specific  structure of the sublinearity $h$  and as one would see, the fact that $q<2$ would be crucially exploited subsequently. 

 We then note that since $q < 2$, therefore for $\gamma$ sufficiently large ( depending on $C, T$ and $2-q$)  and for  $K>2$, we can ensure that
\begin{equation}\label{t2}
C \int\frac{HG}{t}  + C \int  \frac{K}{t^2} HG + (q - 2) \int \frac{HG}{t^{K+1}}(2\gamma K) \leq (q-2) \int \frac{HG}{t^{K+1}} \gamma K
\end{equation}
Now the remaining term 
\[
C \int \frac{H G \left< x \right>^2}{t^2}
\]
is estimated using Lemma \ref{rl} in the following way.
\begin{equation}\label{t4}
\int \frac{HG\left< x \right>^2}{t^2}  \le  C\left( \int \frac{H}{t} e^{-2\gamma - \frac{\frac{b}{2} \left< x \right>^2 + K}{t} }  + 2\gamma \int \frac{H}{t^{K+1}} G  \right)
\end{equation}
We note that \eqref{t4} follows by an application of the  inequality in Lemma \ref{rl} with  $\epsilon = e^{-2\gamma (t^{-K})}, m=1$ and $y=\frac{b}{2} \frac{\left< x \right>^2}{t}$. Using \eqref{t2} and \eqref{t4} in \eqref{k1}, we obtain for some other constant $C$ that the following holds,
\begin{equation}\label{j1}
I_1 \leq C \left( \int \frac{H}{t} e^{-2\gamma - \frac{\frac{b}{2} \left< x \right>^2 + K}{t} }  + 2\gamma \int \frac{H}{t^{K+1}} G \right) +  (q-2) \int \frac{HG}{t^{K+1}} \gamma K
\end{equation}
Consequently,  if  $K$ is chosen large enough, then we can guarantee that,
\begin{equation}\label{j2}
2 C  \gamma \int \frac{H}{t^{K+1}} G  +  (q-2) \int \frac{HG}{t^{K+1}} \gamma K \leq \left(\frac{q}{2} -1 \right) \int  \frac{HG}{t^{K+1}} \gamma K
\end{equation}
Thus by using using \eqref{j2} in \eqref{j1}, we obtain
\begin{equation}\label{j4}
I_1 \leq C  \int \frac{H}{t} e^{-2\gamma - \frac{\frac{b}{2} \left< x \right>^2 + K}{t} } + \left(\frac{q}{2} -1 \right) \int  \frac{HG}{t^{K+1}} \gamma K
\end{equation}
Also using Lemma  \ref{bdd2}, we note that the left hand side in  the corresponding identity in  Lemma \ref{re1} can be bounded from below  by 
\begin{align}\label{j5}
   C_1 K \int (u^2 + |\nabla u|^2 ) G dx
\end{align}
for some universal $C_1>0$.
Finally from \eqref{t0},  \eqref{j4} and \eqref{j5}, we observe that if $K$ is additionally chosen large enough such that
\[
C_1K \geq 4 M
\]
then the following integral  in \eqref{t0}, 
\[
2M \int u^2 G
\]
can be absorbed in the left hand side and consequently the
 desired estimate as claimed in  \eqref{ce1} follows.

\end{proof}
\begin{rmrk}
We cannot stress enough the crucial role  of the inequality \eqref{t4} and its interplay with the sublinear structure  that allows us to establish the estimate \eqref{ce1}.
\end{rmrk}
With the Carleman estimate as in \eqref{ce1} in our hands, we now proceed with the proof of Theorem \ref{main} i).
\vspace{.1in}
\begin{proof}[Proof of Theorem \ref{main} i)]
First we note that it follows from the space like strong unique continuation result  as in Theorem 1.1  in \cite{BM} that $v(\cdot, 0) \equiv 0$.  Moreover from  Step 2 in the proof of Theorem 1.1 in \cite{BM}, it also follows that $v$ vanishes to infinite order in time at $t=0$.  Furthermore from the classical regularity theory as in \cite{Li}, we have that $\nabla^2 u , u_t \in L^{p}_{loc}$ for all $p<\infty$. Thus for $t<0$,   if we  extend $v$ by $0$, the principal coefficients $a^{ij}$ by $a^{ij}(x,0)$ and $W$ by $0$, we note that the extended $v$   continues to solve an equation of the type \eqref{e0} with similar structural assumptions.  Then by using rescaling and translation of the type, 
\[
\tilde v (x,t) = v(rx, r^2 (t-1/2))
\]
for $r$ sufficiently small,   we can ensure that 
\begin{equation}\label{ty1}
|\tilde v(x,t)| \leq C e^{\frac{b}{8} |x|^2}.
\end{equation}
 Subsequently we let  $\tilde v$ be our new  $v$ and thus we are reduced to the situation where   $v\equiv 0$ for $t\leq 1/2$.  Now let $\eta$ be a function of $t$ defined as 
\begin{equation}
\begin{cases}
\eta(t) \equiv 1\ \text{for $t< 3/4$}
\\
\eta(t) \equiv 0\ \text{for $t>7/8$}
\end{cases}
\end{equation}
By  a standard limiting argument ( using cut-offs in space)  as in the proof of Lemma 2.1 in \cite{WZ}, thanks to the Tychonoff type bounds  in  \eqref{ty1}( which ensures integrability of the integrals in \eqref{ce1}  as $|x| \to \infty$ corrresponding to the weight $G$ as in \eqref{g}), the Carleman estimate  in \eqref{ce1} continues to be valid  for $u=\eta v$. Then we note that using \eqref{e0}, we have that $u$ solves \eqref{o1} with 
\[
g= v \eta_t  + h(X, u) - \eta h(X, v)
\]
From the definition of $\eta$ above, it follows that 
\begin{align}\label{bn1}
&g \equiv 0\ \text{ for $t \leq 3/4$ and }\\
& |g| \leq C  (|v| + \sum |v|^{p_i-1})\notag
\end{align}
Now by applying the  Carleman estimate \eqref{ce1} to $u$ we obtain,

\begin{equation}\label{gt1}
K \int ( (\eta v)^2 + | \nabla (\eta v)|^2 ) G dx dt \le C( \int g^2 G dx dt + \int \frac{H(X,\eta v)}{t} e^{-\frac{\frac{b}{2} \left < x \right >^2 + K }{t} -2 \gamma}  dx dt)\end{equation}
 
Now let $l  \in (1/2, 3/4)$. Then by  minorizing the integral on the left hand side in \eqref{gt1} over the region $\{\frac 12 \leq t \leq \frac 34\}$  and by using  \eqref{bn1} we deduce that the following holds,

\begin{align}\label{v1}
& K \int_{\frac 12 \le t \le \frac 34} (v^2 + |\nabla v|^2 ) G dx dt  \\  
&\leq  C \int \frac{H(x,\eta v)}{t} e^{\frac{-\frac{b}{2} \left < x \right > ^2 + K}{t} - 2\gamma}dx dt + C \int_{\frac 34 \le t \le 1} g^2 G dx dt\notag 
\end{align}
Continuing further we obtain,
\begin{align}\label{v2}
& e^{2\gamma(l^{-K}-1)} \int_{\frac 12 \le t \le l}(v^2 + |\nabla v|^2) e^{-\frac{b \left <x \right > ^2 + K}{t}} dx dt \\
& \le C\bigg( e^{-2\gamma}\int_{\frac 12 \le t \le 1}\frac{H(X, \eta v)}{t}e^{\frac{-\frac{b}{2} \left < x \right > ^2 + K}{t}} dx dt\notag \\
& + e^{2\gamma((\frac 34)^ {-K} - 1)} \int_{\frac 34 \le t \le 1} g^2 e^{-\frac{b \left < x \right >^2 + K}{t}}dx dt \bigg)\notag
 \end{align}
 Dividing both sides of \eqref{v2}   by $e^{2\gamma (l^{-K}-1)}$  we find again by using \eqref{bn1},
 \begin{align}\label{last}
&\int_{\frac 12 \le t \le l} (v^2 + |\nabla v|^2) e^{-\frac{b \left < x \right >^2 + K}{t}}dx dt\\
 & \le C e^{-2\gamma l^{-K}} \int_{\frac 12 \le t \le 1}\sum |v|^{p_i}e^{-\frac{\frac{b}{2} \left < x \right >^2 + K}{t}} dx dt\notag \\
&  + e^{2 \gamma((\frac 34)^{-K} - l^{-K})}C \int_{\frac34 \le t \le 1}  ( |v|^2 + \sum |v|^{2(p_i-1)})e^{-\frac{b \left <x \right >^2 + K}{t}}dx dt\notag
\end{align}
Now letting $\gamma \to \infty$ in \eqref{last},  we find that the right hand side goes to $0$ and thus  we can assert that   $v(\cdot, t) \equiv  0$ for $ \frac 12 \le t \le l$. Now by going back to the original $v$ by scaling back, we obtain  that $v(\cdot, t) \equiv 0$ for $0 \leq t \leq \kappa$ for some $\kappa>0$ universal. Therefore  in this way we can spread the zero set and hence conclude that $v \equiv 0$.

\end{proof}

\begin{proof}[Proof of Theorem \ref{main} ii)]
As previously mentioned in the introduction, the proof of this result is via a  geometric variational approach based on a  Weiss type monotonicity formula  which is tailor-made  for the sublinear problem.  The proof is divided into two steps. 

\medskip

\emph{Step 1: A Weiss type monotonicity formula.}
With $G = \frac{1}{|t|^{\frac n2}} e^{-\frac{|x|^2}{4t}}$(note that this $G$ is different from one in \eqref{g}), following  \cite{Po} and \cite{ST}, we let
\vspace{.1in}
\begin{align}\label{W}
&H(R) = \int_{t=R^2}v^2 G dx
\\
& I(R) = R^2 \int_{t=R^2}|\nabla v|^2 G dx- 2R^2 \int_{t=R^2}H(X,v)G dx\notag
 \\
 &  W_{\gamma} (R) = \frac{I(R)}{R^{2\gamma}} - \frac{\gamma}{2R^{2\gamma}} H(R)\notag
\end{align} 

We now make the following claim.

\emph{Claim:} For $\gamma$ sufficiently large depending also on the $L^{\infty}$ norm of $v$, we have that 
\begin{equation}\label{mon}
W_{\gamma}' (R) \geq 0\ \text{for a.e.  $R \in (0,1)$}.
\end{equation}

In order to establish the claim, we compute the derivatives of the quantities involved in the expression of $W_{\gamma}$. We first compute $H'$. In all the computations below, we will be using the Einstein's  notation of summation of repeated indices.   By integrating by parts and by using the equation \eqref{e0}, we note that
\begin{align}\label{h}
    &H'(R) = \int 2 v v_t (2R) G dx + \int v^2 G_t (2R) dx\\
     &= 4R \int v v_t G dx + 2R \int v^2 \Delta G dx\notag\\
     &= 4R \int v \left [-\Delta v - h(X,v) - Wv  \right ]G dx - 2R \int \nabla (v^2) \cdot \nabla G dx\notag \\
     & = 4 R \int \left ( \nabla (vG) \right ) \cdot \nabla v dx - 4R \int v h(X,v) G dx - 4 R \int Wv^2Gdx - 2R \int \nabla (v^2) \cdot \nabla G dx \notag\\
     & = 4R \int |\nabla v|^2 G dx - 4R \int v h(X,v) G dx - 4R \int W v^2 G dx \notag
\end{align}
Also
\vspace{.1in}
\begin{equation}
    I(R) = R^2 \int |\nabla v|^2 G dx - 2R^2 \int H(X,v) Gdx = I_1 + I_2
\end{equation}
\vspace{.1in}
Again by using \eqref{e0} and by integrating by parts, we  find for a.e. $R$,

\begin{align}\label{i1}
    &I_1'(R) = 2R \int |\nabla v|^2 G dx + R^2 \int 2 v_{i} v_{it}(2R) G dx + R^2 \int |\nabla v|^2 G_t (2R) dx \\
    &= 2R \int |\nabla v|^2 G dx + 4R^3 \int v_{i}v_{ti}G dx + 2 R^3 \int |\nabla v|^2 \Delta G dx\notag \\
    &= 2R \int |\nabla v|^2 G dx - 4 R^3 \int \left ( \Delta v G +  \nabla v \cdot \nabla G  \right)v_t dx - 2R^3 \int \nabla (|\nabla v|^2) \cdot \nabla G dx\notag \\
    & = 2R \int |\nabla v|^2 G dx - 4R^3 \int \left [ (-v_t - h(X,v) - Wv ) G  + \nabla v \cdot \nabla G\right ]v_t dx\notag \\
    & + 2R^3 \int 2 v_{i}v_{ij} \frac{x_j}{2t}G dx\ \text{(note that $\nabla G= -\frac{x}{2t} G$)} \notag\\
    &= 2R \int |\nabla v|^2 G dx - 4R^3\int \left ( (-v_t - h(X,v) - Wv ) G + \nabla v \cdot \nabla G \right )v_t dx\notag \\
    & \ \ \ + 2R   \int (v_{i} ( v_{j} x_j)\frac{x_i}{2t} - v_{i}^2 ) G dx - 2R \int \Delta v v_j x_j G dx\notag \\
   &= - 4R^3\int \left( (-v_t - h(X,v) - Wv ) G + \nabla v \cdot \nabla G \right)v_t dx \notag \\
    &  + 4R^3   \int v_{i} \frac{x_i}{2t} v_{j} \frac{x_j}{2t}     G dx - 2R \int (-v_t - h(X, v) - Wv) v_j x_j G dx \notag\\
        &= 4R^3 \int ( v_t + \nabla v \cdot \frac{x}{2t})^2 G dx + 4 R^3 \int ( h(X,v) + W v) \nabla v \cdot \frac{x}{2t} G dx\notag \\
    & + 4R^3 \int ( (h(X,v) + Wv) v_t G dx\notag\\
    &= 4R^3 \int (v_t + \nabla v \cdot \frac{x}{2t} + \frac 12 Wv)^2 G dx - R^3 \int W^2 v^2 G dx + 4 R^3 \int ( h(X,v)) ( \nabla v \cdot \frac{x}{2t} + v_t) Gdx\notag\end{align}
    We note that  in  the intermediate  computations for $I_	1'(R)$ in \eqref{i1} above,  the term $v_{it}$ appears  and the solution $v$ may not possess that much of Sobolev regularity in general.  However  such formal computations can be justified by approximating $v$ by solutions to smoother problems and a limiting argument which crucially uses  the fact that $v_t, \nabla v, \nabla^2 v  \in L^{2}_{loc}$  with universal bounds depending on the  $L^{\infty}$ norm of $v$.
\vspace{.1in}
Likewise
\vspace{.1in}
\begin{align}\label{i2}
    & I_2'(R)  = -4R \int H(X,v) G dx - 4 R^3 \int H_t (X,v) G dx - 4 R^3 \int h(X,v) v_t G dx - 4 R^3 \int H(X,v) G_t dx \\
    & = -4R \int H(X,v) G dx - 4 R^3 \int H_t (X,v) G dx - 4 R^3 \int h(X,v) v_t G dx - 4R^3 \int h(X, v) \nabla v \cdot \frac{x}{2t} G dx \notag\\
    & - 4R^3 \int \nabla_x H \cdot \frac{x}{2t} G dx\notag \end{align}
Thus from \eqref{i1} and \eqref{i2} we observe  that the integral  $4 R^3 \int_{t=R^2}  ( h(X,v)) ( \nabla v \cdot \frac{x}{2t} + v_t) Gdx$ gets cancelled  in the expression of $I'(R)$.  Therefore using \eqref{h}, \eqref{i1} and \eqref{i2} we obtain, 
\begin{align}\label{w1}
   &  W_{\gamma}'(R)  = \frac{I'(R)}{R^{2\gamma}} - \frac{2\gamma I(R)}{R^{2\gamma+1}} - \frac{\gamma}{2R^{2\gamma}} H'(R) + \frac{\gamma^2}{R^{2\gamma+1}} H(R) \\
    &= \frac{4R^3}{R^{2\gamma}} \int_{t=R^2} (v_t + \nabla v \cdot \frac{x}{2t} + \frac 12 Wv)^2 G dx - \frac{R^3}{R^{2\gamma}} \int_{t = R^2}W^2 v^2 G dx - \frac{4R}{R^{2\gamma}} \int _{t = R^2} H(X,v) Gdx\notag \\ 
    & \ \ \ \ - \frac{4R^3}{R^{2\gamma}} \int_{t=R^2} H_t G dx - \frac
    {4R^3}{R^{2\gamma}}\int_{t= R^2} \nabla _x H \cdot \frac{x}{2t} G dx - \frac{2\gamma R^2 }{R^{2\gamma+1}} \int_{t=R^2} |\nabla v|^2 G dx\notag  \\ 
    & \ \ \ \ + \frac{4\gamma R^2}{R^{2\gamma+1}} \int_{t= R^2} H(X,v) G dx - \frac{4\gamma R}{2R^{2\gamma }} \int_{t=R^2}|\nabla v|^2 G dx + \frac{4\gamma R}{2R^{2\gamma }} \int_{t=R^2} vh(X,v) G dx\notag \\
    & \ \ \ \ + \frac{4\gamma R}{2R^{2\gamma}} \int_{t=R^2} Wv^2 Gdx + \frac{\gamma^2}{R^{2\gamma+1}}\int_{t=R^2}v^2 G dx\notag
    \end{align} 
    
    Now again from \eqref{e0} and divergence theorem, we note that the integral $\int_{t=R^2} |\nabla v|^2 G dx$ can be rewritten as 
    \begin{equation}\label{int1}
    \int_{t=R^2} |\nabla v|^2 G dx=  \int_{t=R^2} ( v_t  + Wv + h(X, v) + \nabla v \cdot \frac{x}{2t} ) v G dx
    \end{equation} 
     Using \eqref{int1} in \eqref{w1}, we find
    \begin{align*}
    & W_{\gamma}'(R) = \frac{1}{R^{2\gamma-1}}\int_{t=R^2} \bigg[[(2R)(v_t + \nabla v \cdot \frac{x}{2t} + \frac 12 Wv)]^2 + (\frac{\gamma v}{R})^2 \bigg] G dx \\
    & \ \ \ \ + \frac{4\gamma }{R^{2\gamma-1}} \int_{t=R^2} ( -v_t - Wv - h(X,v) - \nabla v \cdot \frac{x}{2t}) vG dx\\
    & - \frac{R^3}{R^{2\gamma}} \int_{t = R^2}W^2 v^2 G dx - \frac{4R}{R^{2\gamma}} \int _{t = R^2} H(X,v) Gdx \\ 
    & \ \ \ \ - \frac{4R^3}{R^{2\gamma}} \int_{t=R^2} H_t G dx - \frac
    {4R^3}{R^{2\gamma}}\int_{t= R^2} \nabla _x H \cdot \frac{x}{2t} G dx   \\ 
    & \ \ \ \ + \frac{4\gamma R^2}{R^{2\gamma+1}} \int_{t= R^2} H(X,v) G dx + \frac{4\gamma R}{2R^{2\gamma}} \int_{t=R^2} vh(X,v) G dx \\
    & \ \ \ \ + \frac{4\gamma R}{2R^{2\gamma}} \int_{t=R^2} Wv^2 Gdx \\
    &= \frac{1}{R^{2\gamma-1}} \int _{t = R^2} \bigg[ (2R) (v_t + \nabla v \cdot \frac{x}{2t} + \frac 12 Wv ) - \frac{\gamma v}{R}\bigg]^2 G dx - \frac{R^3}{R^{2\gamma}} \int_{t = R^2} W^2 v^2 G dx \\ 
    & \ \ \ \ - \frac{4R}{R^{2\gamma}} \int _{t = R^2} H(X,v) Gdx \\ 
    & \ \ \ \ - \frac{4R^3}{R^{2\gamma}} \int_{t=R^2} H_t G dx - \frac
    {4R^3}{R^{2\gamma}}\int_{t= R^2} \nabla _x H \cdot \frac{x}{2t} G dx   \\ 
    & \ \ \ \ + \frac{4\gamma R^2}{R^{2\gamma+1}} \int_{t= R^2} H(X,v) G dx - \frac{2\gamma }{R^{2\gamma-1}} \int_{t=R^2} vh(X,v) G dx
    \end{align*}
    Now using the bounds for $H_t$ and $\nabla_x H$ as in \eqref{str}, it follows that for $\gamma$ sufficiently large depending on the bounds in \eqref{str} and \eqref{w}, we have that for some universal $C>0$, 
    \begin{align}\label{w40}
&W_{\gamma}'(R) \ge - \frac{C}{R^{2\gamma -1}} \int_{t=R^2} v^2G dx  + \frac{4\gamma}{R^{2\gamma-1}} \int_{t= R^2} H(X,v) G dx - \frac{2\gamma }{R^{2\gamma-1}} \int_{t=R^2} vh(X,v) G dx \\
&- \frac{C}{R^{2\gamma -1 }} \int_{t=R^2}H(X,v)G dx\notag\end{align}
Then by using $vh(X, v) \leq  q H(X, v)$, we obtain from \eqref{w40} that the following holds,
  \begin{align}\label{w4}
&W_{\gamma}'(R) \ge - \frac{C}{R^{2\gamma -1}} \int_{t=R^2} v^2G dx  + \frac{1}{R^{2\gamma-1}} \int_{t= R^2} (2 \gamma (2-q)  - C ) H(X,v) G dx \end{align} Now since $ q<2$. therefore by choosing $\gamma$ large enough, we can ensure that
\begin{align}\label{w5}
& W_{\gamma}'(R) \geq  \frac{1}{R^{2\gamma-1} }\int_{t=R^2}( C_1 \gamma  H(x,v)  - C v^2) G dx
\end{align}
At this point, we use the fact that since $v$ is bounded, therefore from \eqref{str1} it follows that   
\begin{equation}\label{cr}
H(x, v) \geq  c_0 |v|^q
\end{equation}
where $c_0$ depends on the $L^{\infty}$ norm of $v$ (this is precisely where we use the boundedness of $v$).

Using \eqref{cr} in  \eqref{w5}, we deduce  that the following holds for a new constant $C_2$ depending also on $c_0$, 
\begin{align}\label{w7}
& W_{\gamma}'(R) \geq  \frac{1}{R^{2\gamma-1} }\int_{t=R^2} |v|^q ( C_2 \gamma    - C |v|^{2-q}) G dx
\end{align}
Now since $v$ is bounded, therefore by choosing $\gamma$ sufficiently large depending also on the $L^{\infty}$ norm of $v$, we can ensure that 
\begin{equation}\label{boundedness}
C_2 \gamma - C|v|^{2-q} \geq 0
\end{equation}
and thus we can assert that the Weiss type monotonicity  as  claimed in  \eqref{mon} holds.

\emph{Step 2: Conclusion.}
Assume on the contrary that $v$ is not identically zero in $\Rn \times (0,1 )$. Then there exists $R>0$ such that $v(\cdot, R^2) \neq  0$. We now  choose  $\gamma >0$ large enough such that  simultaneously both \eqref{mon} as well as 
 \[W_{\gamma}(R) < 0\]  hold.  Then from the  monotonicity of $W_{\gamma}$, we must have that $W_{\gamma}(0+) < 0$. However since $v$ vanishes to infinite order at $(0,0)$ in the sense of \eqref{vp1}, it follows from the expression of $W_{\gamma}$ as in \eqref{W}   that $W_{\gamma}(0+) \geq 0$ (see for instance Lemma 1.4 in \cite{Po}). This leads to a contradiction and thus finishes the proof of the Theorem.
\end{proof}


\begin{thebibliography}{99}

\bibitem{Al}
F. J. Almgren, Jr., \emph{Dirichlet's problem for multiple valued functions and the regularity of mass minimizing integral currents. Minimal submanifolds and geodesics}, (Proc. Japan-United States Sem., Tokyo, 1977), pp. 1-6, North-Holland, Amsterdam-New York, 1979.

\bibitem{AKS}
Aronszajn, N. , Krzywicki, A \& Szarski, J. \emph{A unique continuation theorem for exterior differential forms on Riemannian manifolds},Ark. Mat. \textbf{4}~ 1962 417-453 (1962).


\bibitem{BGM}
A. Banerjee, N. Garofalo \& R. Manna, \emph{Carleman estimates for Baouendi-Grushin operators with applications to quantitative uniqueness and strong unique continuation}, 	arXiv:1903.08382, to appear in Applicable Analysis.
\bibitem{BM}
A. Banerjee \& R. Manna, \emph{Space like strong unique continuation for sublinear parabolic equations}, arXiv:1812.10246, to appear in Journal of London Mathematical Society.


\bibitem{Car}
T. Carleman, \emph{Sur un probleme d'unicite pur les systemes d'equations aux derivees partielles a deux variables independantes}, Ark. Mat., Astr. Fys. \textbf{26}~ (1939). no. 17, 9 pp 


\bibitem{Ch}
X. Chen, \emph{A strong unique continuation theorem for parabolic equations}, Math. Ann., \text{311}~(1998), 603-630. 





\bibitem{EF}
L. Escauriaza \& F. Fernandez, \emph{Unique continuation for parabolic operators. (English summary)}, Ark. Mat. \textbf{41}~ (2003), no. 1, 35-60.

\bibitem{EFV}
L. Escauriaza, F. Fernandez \& S. Vessella, \emph{Doubling properties of caloric functions}, Appl. Anal. \textbf{85}~ (2006), no. 1-3, 205-223.




\bibitem{GL}
N. Garofalo \& F. Lin, \emph{Monotonicity properties of variational integrals, $A_p$ weights and unique continuation}, Indiana Univ. Math. J. \textbf{35}~(1986),  245-268.


\bibitem{Ho}
L. H\"ormander, \emph{Uniqueness theorems for second order elliptic differential equations}. Comm. Partial Differential Equations \textbf{8}~(1983), no. 1, 21-64. 

\bibitem{J}
F. Jones, \emph{A fundamental solution for the heat equation which is supported in a strip},  J. Math. Anal. Appl. \textbf{60}~ (1977) 

 \bibitem{JK}
  D. Jerison \& C. Kenig, \emph{Unique continuation and absence of positive eigenvalues for Schrodinger operators}, Ann. of Math. (2) 121 (1985), no. 3, 463-494.  
  







 \bibitem{KT0}
  H. Koch \& D. Tataru, \emph{Carleman estimates and unique continuation for second-order elliptic equations with nonsmooth coefficients}, Comm. Pure Appl. Math. 54 (2001), no. 3, 339-360. 

\bibitem{KT}
H. Koch \& D. Tataru, \emph{Carleman estimates and unique continuation for second order parabolic equations with nonsmooth coefficients  (English summary) }, 
Comm. Partial Differential Equations \textbf{34}~ (2009), no. 4-6, 305-366. 


\bibitem{L}
F. Lin, \emph{A uniqueness theorem for parabolic equations}, Comm. Pure Appl. Math., \textbf{43}~(1990), 127-136. 

\bibitem{Li}
G. Lieberman, \emph{Second order parabolic differential equations},  World Scientific Publishing Co., Inc., River Edge, NJ, 1996. xii+439 pp. ISBN: 981-02-2883-X.

\bibitem{LM}
J. Lions \& B. Malgrange, \emph{Sur lunicite retrograde dans les problemes mixtes paraboliques}. (French)
Math. Scand. \textbf{8}~ (1960), 277?286.

\bibitem{LO}
E. M. Landis \& O. A. Oleinik, \emph{ Generalized analyticity and some related properties of solutions of elliptic and parabolic equations}, Russ. Math. Surv. \textbf{29}~(1974), 195-212. 


\bibitem{Mi}
S. Mizohata, \emph{Unicite du prolongement des solutions pour quelques operateurs differentiels paraboliques}, Mem. Coll. Sci. Univ. Kyoto. Ser. A. Math., \textbf{31}~219-239. 


  
  \bibitem{Po}
C. C. Poon, \emph{Unique continuation for parabolic equations}, Comm. Partial Differential Equations \textbf{21}~ (1996), no. 3-4, 521-539.  

\bibitem{PW}
E. Parini \& T. Weth, \emph{ Existence, unique continuation and symmetry of least energy nodal solutions to sublinear Neumann problems}, Math. Z, \textbf{280}~(2015), 707-732.
\bibitem{Ru}
A. Ruland, 
\emph{Unique Continuation for Sublinear Elliptic Equations Based on Carleman Estimates}, 	J. Differential Equations \textbf{265}~ (2018) 6009-6035.

\bibitem{So}
C. Sogge, \emph{A unique continuation theorem for second order parabolic differential operators}, Ark. Mat. \text{28}~(1990), 159-182. 


\bibitem{ST}
N. Soave \& S. Terracini, \emph{The nodal set of solutions to some elliptic problems: sublinear equations, and unstable two-phase membrane problem}, 	arXiv:1802.02089 
\bibitem{SW}
N. Soave \& T. Weth, \emph{The unique continuation property of sublinear equations},  SIAM J. Math. Anal. \textbf{50}(4)~ (2018), 3919-3938.

\bibitem{Vaz}
J. Vazquez, \emph{The porous medium equation: mathematical theory}, Oxford University press, 2007. 
\bibitem{V}
S. Vessella, \emph{Unique continuation properties and quantitative estimates of unique continuation for parabolic equations}, Handbook of differential equations: evolutionary equations. Vol. V, 421-500. Handb. Differ. Equ., Elsevier/North-Holland, 2009.
\bibitem{We}
G. Weiss, \emph{A homogeneity improvement approach to the obstacle problem}, Invent.
Math. \textbf{138}~ (1999), no. 1, 23-50.

\bibitem{WZ}
J. Wu \& L. Zhang, \emph{Backward uniqueness for general parabolic operators in the whole space}, (English summary)
Calc. Var. Partial Differential Equations \textbf{58}~ (2019), no. 4, 19pp.

\bibitem{WZ1}
\bysame, \emph{Backward uniqueness of parabolic equations with variable coefficients in a half space}, 	Commun. Contemp. Math. \textbf{18}~ (2016), no. 1, 38 pp.


\bibitem{Y}
H. Yamabe, \emph{A unique continuation theorem of a diffusion equation}, Ann. of Math.(2)  \textbf{69}~(1959), 462-466. 


\end{thebibliography}
\end{document}